\documentclass[11pt]{article}
\usepackage[backref=page,colorlinks=true]{hyperref}
\usepackage{amsmath,amsthm,amsfonts,amssymb,amscd,url}
\usepackage[capitalize]{cleveref}
\usepackage{graphics,xcolor}
\numberwithin{equation}{section}
\input{xy} \xyoption{all}
\usepackage[latin1]{inputenc}

\usepackage{enumerate}
\usepackage{hyperref}
\usepackage{epsfig} 
\input{xy} \xyoption{all} 
\usepackage{mathrsfs}

\def\diag {\mathrm{\diag}\, }

\def \cC{{\mathscr C}}

\def\qand{\quad \text{and}\quad}

\def\N{\mathbb N}

\def\Z{\mathbb Z}
\def\cU{\mathcal U}

\def\cC{\mathcal C}
\def\cU{\mathcal U}

\def\su{{\mathsf {\mathbf u}}}

\def\sw{{\mathsf {w}}}
\def\ss{{\mathsf {s}}}

\usepackage{mathtools}
\def\trans{\cap\kern-0.7em|\kern0.7em}

\newtheorem{proposition}{Proposition}[section]

\newtheorem*{theorem2}{Theorem}
\newtheorem{definition}[proposition] {Definition}
\newtheorem*{definition2}{Definition}
\newtheorem{lemma}[proposition] {Lemma}
\newtheorem{sublemma}[proposition] {Sublemma}
\newtheorem{fact}[proposition]{Fact}
\newtheorem*{corollary}{Corollary}
\newtheorem{corol}[proposition]{Corollary}
\theoremstyle{remark}

\newtheorem{remark}[proposition]{Remark}
\newtheorem*{question}{Question}

\DeclareTextFontCommand{\emph}{\em\bf}
\begin{document}
\title{Robust complex heterodimensional cycles}
\author{S\'ebastien Biebler}

\maketitle
\begin{abstract}
A diffeomorphism $f$ has a heterodimensional cycle if it displays two (transitive) hyperbolic sets $K$ and $K'$ with different indices such that the unstable set of $K$ intersects the stable one of $K'$ and vice versa.  We prove that it is possible to find robust heterodimensional cycles for families of polynomial automorphisms of $\mathbb{C}^3$. The proof is based on Bonatti-D\'iaz blenders.
\end{abstract}
\tableofcontents

\section{Introduction}

The dynamics of uniformly hyperbolic systems such as the horseshoe introduced by Smale is particularly well understood. Such systems were originally conjectured to be dense in the 1960's. However it was soon realized that this is not true. There are two main mechanisms that yield robustly non-hyperbolic behaviors: robust heterodimensional cycles (obtained by Abraham and Smale \cite{AS}) and robust homoclinic tangencies (this is the celebrated Newhouse phenomenon \cite{NEW1,NEW2}).  Later Palis \cite{PAL1,PAL2} conjectured that diffeomorphisms exhibiting either a heterodimensional cycle or a homoclinic tangency  are dense in the complement of the closure of the hyperbolic ones.\\

In this paper, we are interested in extending the phenomenon of robust heterodimensional cycles to the setting of complex dynamics. \\

A compact invariant set $K$ for a diffeomorphism $f$ of a manifold is a {\bf hyperbolic basic set} if $K$ is hyperbolic, transitive and locally maximal. Its {\bf index} denotes the dimension of the unstable manifold of any of its points (all unstable manifolds of points of $K$  have the same dimension since the hyperbolic set $K$ is transitive). 

\begin{definition2}
Let $f$ be a diffeomorphism of a manifold of dimension $3$ or higher having two hyperbolic basic sets $K$ and $K'$. We say that $f$  has a {\bf heterodimensional cycle} associated to $K$ and $K'$ if:
\begin{enumerate} 
\item the indices of $K$ and $K'$ are not equal,
\item the stable sets $W^s(K)$ and $W^s(K')$ intersect respectively the unstable sets $W^u(K')$ and $W^u(K)$.
\end{enumerate}
\end{definition2}

Bonatti and D\'iaz showed in \cite{BD3} that heterodimensional cycles of coindex 1 can be $C^1$-stabilized. In a recent breakthrough, Li and Turaev \cite{LT1} solved this $C^r$-persistence problem in any regularity $r=2, \cdots, \infty, \omega$ (one can also refer to \cite{LT2,LLST}). \\

Let $\mathrm{Aut}(\mathbb{C}^3)$ and $\mathrm{Aut}_d(\mathbb{C}^3)$ be the spaces of holomorphic automorphisms and polynomial automorphisms of degree at most $d$ of $\mathbb{C}^3$. A polynomial automorphism of degree at most $d$ is an element of $\mathrm{Aut}(\mathbb{C}^3)$ whose component functions are polynomials of degree at most $d$. We endow these spaces with the topology induced by local uniform convergence of the map and its inverse. \\

Our main result shows that we can find robust heterodimensional cycles for polynomial automorphisms of $\mathbb{C}^3$: 

\begin{theorem2} \label{main} 
There exists a polynomial automorphism $f$ of $\mathbb{C}^3$ such that every $g \in \mathrm{Aut}(\mathbb{C}^3)$ sufficiently close to $f$ displays  a heterodimensional cycle associated to the hyperbolic continuations of a hyperbolic basic set of index $2$ and a saddle periodic point of index $1$.
\end{theorem2}

\begin{corollary} \label{mainc} 
There exists an integer $d \ge 2$ and a nonempty open set $\cU_d \subset \mathrm{Aut}_d(\mathbb{C}^3)$ such that every $g \in \cU_d$ displays a heterodimensional cycle   associated to  the hyperbolic continuations of a hyperbolic basic set of index $2$ and  a saddle periodic point of index $1$.
\end{corollary}

In particular, our results can be seen as an analogous of Buzzard's work \cite{BUZ} on complex Newhouse phenomenon (see also \cite{DUJ2,ALZ}) for heterodimensional cycles instead of tangencies. \\

The main tool to obtain persistent heterodimensional cycles is a type of hyperbolic sets with very special fractal properties called {\bf blender}. Introduced by Bonatti and D\'iaz in \cite{BD1}, blenders admit several definitions. Here is one:

\begin{definition2}
A hyperbolic basic set $K_f$ for a diffeomorphism $f$ of a manifold of dimension $3$ or higher is a {\bf $cs$-blender} if there is a nonempty open set $\Delta$ of embedded $d$-dimensional disks, with $d$ smaller than the index of $K_f$, such that for every $C^1$-perturbation $g$ of $f$, the stable set $W^s(K_g)$ of the hyperbolic continuation $K_g$ of $K_f$ for $g$ intersects any disk in $\Delta$.
\end{definition2} 

We will use the blender property to obtain  in $\mathbb{C}^3$ robust intersections between the one-dimensional unstable manifold of a saddle periodic point of index $1$, and the one-dimensional stable set of a hyperbolic basic set of index $2$, which was not expected a priori. Notice that blenders have many other powerful applications in dynamical systems, see for example \cite{ACW,AST,BE1,BE2,BIE1,BIE2,BD2,DUJ1,TA, GTV}. \\ 

Remark that the degree of the maps we obtain is unknown since we used a Runge approximation argument. It is then natural to ask:

\begin{question}
Is it possible to give an explicit degree $d$ in the main Corollary  ? 
\end{question}

This paper is organized as follows. We define successively a polynomial $p$ of $\mathbb{C}$, then a H\'enon map $H$ of $\mathbb{C}^2$ and finally  a polynomial automorphism $F_1$ of $\mathbb{C}^3$ which is a skew-product  over $H$ displaying both a horseshoe of index $2$ and a saddle point of index $1$ (Section 2). Then we show that this horseshoe satisfies an open covering property on the third coordinate and hence the blender property (Section 3). Finally, we create an initial heterodimensional cycle and use the blender property to make it robust (Section 4). \\

{\bf Aknowledgements.} The author would like to thank Pierre Berger, Romain Dujardin, Johan Taflin and Gabriel Vigny for many invaluable comments and discussions. 

This research was partially supported by the ANR project PADAWAN, ANR-21-CE40-0012-01.

\section{A skew-product polynomial automorphism}

\subsection{One dimensional dynamics}

In this subsection, we introduce  a polynomial $p$ displaying  a repelling Cantor set and a sink. 
 
 Let us start with a few notation.
For every $z \in \mathbb{C}$ and $r>0$, let $\mathbb{D}(z,r) \subset \mathbb{C}$ be the  closed disk of center $z$ and radius $r$. We set:
\[ \mathbb{D} := \mathbb{D}(0,1) \, .  \]
 We fix once and for all $\eta := 10^{-4}$ and  also define the following disks:
\[ \mathbb{D}_0 := \mathbb{D}(1/4, \eta  ) \, ,  \quad  \mathbb{D}_1 := \mathbb{D}(i/4,\eta ) \, ,  \quad \mathbb{D}_2 := \mathbb{D}(-1/4, \eta  ) \, , \quad   \mathbb{D}_3 := \mathbb{D}(-i/4, \eta  ) \, ,  \]
\[   \mathbb{D}_4 := \mathbb{D}(3,1) \, , \qand   \mathbb{D}_5 := \mathbb{D}(3, \eta  ) \, .  \]
For  $ j = 0,1,2,3$, let $\ell_j$ be the expanding affine map of linear coefficient $1/\eta$ sending $\mathbb{D}_j$ bijectively onto $\mathbb{D}$. Let $\ell_4$ be  the  affine contraction of linear coefficient $\eta $ sending $\mathbb{D}_4$ bijectively onto $\mathbb{D}_5$. In particular, $\ell_4$ displays an attracting fixed point equal to $3$ with multiplier equal to $\eta $.

\begin{proposition} \label{polynomep} 
For every $\varepsilon>0$, there exists a polynomial map $p : \mathbb{C} \rightarrow \mathbb{C}$ such that:
\[ \forall \,  0 \le j \le   4 \, , \forall z \in \mathbb{D}_j \, , \quad | p(z) - \ell_j (z) | < \varepsilon \qand | p' (z) - \ell'_j (z) | < \varepsilon \, . \] 
\end{proposition}

\begin{proof}
This is an immediate consequence of the Runge's theorem. 
\end{proof}

\begin{remark} \label{remarkpolynomepfixedpt}
The polynomial $p$ has an  attracting fixed point close to $3$ with multiplier close to $\eta $.
\end{remark}

\begin{corol} \label{corop}
For every $\delta>0$ small enough,  $p^{-1} ( \mathbb{D}(\delta,1) ) $ contains (at least) four connected components $\Delta_{\delta,j}$, $j \in \{0,1,2,3\}$,  which are topological disks close to $\mathbb{D}_j$ for the Hausdorff topology. Moreover each $\Delta_{\delta,j}$ depends continuously on $\delta$. 
\end{corol}

\subsection{Two dimensional dynamics}

Here we define a H\'enon map $H$ from the polynomial $p$ defined in the latter subsection.
We take $\varepsilon>0$ small and a polynomial $p$ as given by \cref{polynomep}.  

Let us take $b \in \mathbb{C}^*$ and let us consider the H\'enon map:
\[ H : (z,w) \in \mathbb{C}^2 \mapsto (p(z)+bw,z) \in \mathbb{C}^2  \, . \]
Recall that for $b \neq 0$ the map $H$ is a polynomial automorphism of $\mathbb{C}^2$. 
When $|b|>0$ is small enough, by \cref{corop},  the set $H^{-1}(\mathbb{D}^2) \cap \mathbb{D}^2$ has (at least) four connected components $U_j$, with: 
 \[  U_j :=   \bigcup_{w \in \mathbb{D}} \Delta_{-bw,j} \times \{w\}  \, , \,    \quad  j \in \{0,1,2,3\}  \, . \]
Recall that each $\Delta_{-bw,j}$ is a topological disk close to $\mathbb{D}_j$ for the Hausdorff topology. Therefore $U_j$ is close to $\mathbb{D}_j \times \mathbb{D}$. 

\begin{proposition} \label{henon}
There exists $b_o >0$ such that for any $b \in \mathbb{C}^*$ satisfying  
 $|b|<b_o$, the H\'enon map $H$ 
displays a hyperbolic  basic set of saddle type 
\[ K_o := \bigcap_{n \in \mathbb{Z}} H^{n}(\bigsqcup_{ 0 \le j \le 3 } U_j) \]
 and an attracting fixed point $S_o$ close to $(3,3)$ whose basin contains $\mathbb{D}(3,1)^2$.  
\end{proposition} 

\begin{proof} 
{\bf We first show that $K_o$ is a hyperbolic basic set of saddle type} (see Theorem 3.1 in \cite{FS} for the proof of a very similar result). 
The set $K_o$ is clearly compact, invariant by $H$ and locally maximal. 
We are going to show that $K_o$ is a horseshoe in the sense of Definition 6.5.2 of \cite{KH}.
The image $V_j := H(U_j)$ is equal to:
\[ V_j = H(U_j) =     \{ (p(z) + bw, z) \, | \, w \in \mathbb{D}  \, , \, z \in \Delta_{-bw,j} \} \, , \]
which is close to the curve $\{ (p(z),z) \, | \, z \in \mathbb{D}_j  \}$ when $b$ is small. Notice that the second coordinate projection of $U_j$ is equal to $\mathbb{D}$. Also for every $w \in \mathbb{D}$, the first coordinate projection of  the restriction of $H$ to  $U_j \cap (\mathbb{D} \times \{w\}) =  \Delta_{-bw,j} \times \{w\} $ is a bijection onto  $p( \Delta_{-bw,j}) + bw = \mathbb{D}$. Therefore the connected component $U_j$ is full in the sense of Definition 6.5.1 of \cite{KH}.   
 We also have:
\[ U_j \subset   \mathrm{int}  \,   \mathbb{D}  \times  \mathbb{D}  \qand  V_j \subset   \mathbb{D} \times    \mathrm{int}   \,  \mathbb{D}  \, . \]
We now define the two following constant cone fields:
 \[ \chi^{u} = \{(v_{1},v_{2}) \in \mathbb{C}^{2} :   |v_{2}|  \le 10^{-3} \cdot  |v_{1}|   \}    \qand  \chi^{s} = \{ (v_{1},v_{2}) \in \mathbb{C}^{2} :  |v_{1}| \le  10^{-3} \cdot  |v_{2}|  \} \, .   \]
Notice that at any point  $(z,w) \in \bigsqcup_{0 \le j \le 3}  U_j$ the differential $DH$ is equal to:
\[\begin{pmatrix}
p' (z) & b \\
1 & 0 
\end{pmatrix}
\]
with $p' (z)$ close to $1/\eta  = 10^4$ by \cref{polynomep}. 
Therefore, if $b$ is small enough, any  non zero vector in  $\chi^{u}$ is sent into $\mathrm{int} \, \chi^{u}$ and expanded by a factor larger than $10^3$ by $D H$. 
Also  at any point  $(z,w) \in \bigsqcup_{0 \le j \le 3} V_j$ the differential $DH^{-1}$ is equal to:
\[\begin{pmatrix}
0 & 1 \\
1/b & - p' (w) /b  
\end{pmatrix}
\]
with $p' (w)$  again close to $1/\eta  = 10^4$. 
Therefore, if $b$ is small enough, any  non zero vector in $\chi^{s}$  is sent into $\mathrm{int} \,  \chi^{s}$ and expanded by a factor  larger than $10^3$ by $DH^{-1}$.
 Then $K_o$ is a horseshoe in the sense of Definition 6.5.2 of \cite{KH}. According to the discussion following this definition, $K_o$ is also hyperbolic (this is an immediate consequence of the cone field criterion (Corollary 6.4.8 of \cite{KH}) and topologically conjugate to a shift. In particular, it is transitive and therefore a hyperbolic basic set (of saddle type). 
 \medskip

{\bf Now we show that $H$ displays a sink.}
Recall that $p$ displays a sink of period 1 close to 3 by \cref{remarkpolynomepfixedpt}. 
Hence the map $(z,w) \mapsto (p(z),z)$ has a fixed point close to $(3,3)$.
Therefore by the implicit function theorem, for $b$ small, the map $H$ displays the continuation of a fixed point $S_o$ close to $(3,3)$. The  differential of $H$ at any $(z,w) \in \mathbb{D}(3,1)^2$ is equal to: 
\[\begin{pmatrix}
p' (z) & b \\
1 & 0 
\end{pmatrix}
\]
where $p' (z)$ is close to $\eta  = 10^{-4}$ and $b$ is small. 
Therefore the two eigenvalues of the  differential of $H$ at $S_o$ are smaller than 1 in modulus and thus $S_o$ is a sink for $H$.
Since $p(\mathbb{D}_4)$ is close to $\mathbb{D}_5$ and since $b$ is small,     $\mathbb{D}(3,1)^2$ is sent into itself by $H$. Also the  differential of $H^2$ at any $(z,w) \in \mathbb{D}(3,1)^2$ is close to: 
\[\begin{pmatrix}
10^{-8} & 0 \\
10^{-4} & 0 
\end{pmatrix}
\]
Then, by the mean value inequality,  the distance of $H^2(z,w)$ to $S_o$ is smaller than half  the distance of $(z,w)$ to $S_o$, hence $H^n(z,w)$ tends to $S_o$ when $n$ is large. 
Thus the basin of $S_o$ contains $\mathbb{D}(3,1)^2$. 
\end{proof}

\subsection{Three dimensional dynamics}

In this subsection, we define a polynomial automorphism $F_1$ of $\mathbb{C}^3$ which is a skew-product over the H\'enon map $H$. 
We start with the following one-dimensional intermediate result: 

\begin{lemma}  \label{vartheta}
For every $\zeta>0$, there exists a polynomial  map $q : \mathbb{C} \rightarrow \mathbb{C}$ such that:
\[ \forall \,  0 \le j \le 3 \, , \forall  \,  z \in \mathbb{D}_j \, , \quad | q(z) - \frac{1}{10}  e^{i \frac{j \pi}{2}} | < \zeta  \qand | q' (z)  | < \zeta \, ,  \] 
\[ \forall \,    z \in \mathbb{D}_4 \, , \quad | q(z) + \frac13 | < \zeta \qand | q' (z)  | < \zeta \, . \]
\end{lemma} 

\begin{proof} 
This is again an immediate application of the Runge's theorem.
\end{proof}

Let us fix once and for all:
\[ \lambda := \frac{10}{9} \, .  \]

We take $\zeta>0$ small and a polynomial $q$ as given by \cref{vartheta}. 
We consider the H\'enon map $H$ given by \cref{henon}.
The following map  is a  polynomial automorphism of $\mathbb{C}^3$ for $b \neq 0$: 
\[ F_1 : (z,w,t) \in \mathbb{C}^3 \mapsto (H(z,w) , \lambda t + q(z)) = (p(z) + bw ,z ,  \lambda t + q(z)) \, .  \] 

We set:
\[ J := \{0,1,2,3\} \, . \] 
Notice that $\mathbb{D}^3 \cap F_1^{-1}(\mathbb{D}^3)$ displays  (at least) four connected components $R^j (F_1)$,  $j \in J$, close to:
\[ \mathbb{D}_j \times \mathbb{D} \times \frac1\lambda  \cdot  ( \mathbb{D} - \frac{1}{10}  e^{i \frac{j \pi}{2}} ) \, . \]   
For every  $F$ close to  $F_1$, for every $j \in J$, we denote by $R^j = R^j (F)$ the connected component of $\mathbb{D}^3 \cap F^{-1}(\mathbb{D}^3)$ which is the continuation of $R^j(F_1)$.  Its image $F(R^j)$ is close to $\{ (p(z) ,z) \, | \, z \in \mathbb{D}_j \}  \times \mathbb{D}$. \medskip

Let us consider  the two following constant cone fields:
 \[ C^{u} :=   \{ (v_{1},v_{2},v_{3}) \in \mathbb{C}^{3} : |v_{2}| \le  10^{-3} \cdot   \|(v_1,v_3)\|_2 \}    \]  
\[  \qand  C^{s} :=   \{ (v_{1},v_{2},v_{3}) \in \mathbb{C}^{3} :   \|(v_1,v_3)\|_2   \le 10^{-3} \cdot |b| \cdot  |v_{2}|\}   \, . \] 

\begin{proposition} \label{coneproperties}
For every  automorphism  $F$ close to $F_1$,   any  non zero vector in  $C^{u}$  is sent into $\mathrm{int} \, C^{u}$ and expanded by a factor  larger than $\frac{1+\lambda}{2}$ by $D_{(z,w,t)} F$ for any  
 $(z,w,t) \in \bigsqcup_{j \in J} R^j$.
 
  Moreover  any  non zero vector   $C^{s}$  is sent into $\mathrm{int} \,   C^{s}$ and expanded by a factor  larger than $10^3$ by $D_{(z,w,t)} F^{-1}$ for any $(z,w,t) \in \bigsqcup_{j \in J} F(R^j)$. 
\end{proposition} 

\begin{proof}
Note that the differentials of $F$ and its inverse are close on $\bigsqcup_{j \in J} R^j$ and $\bigsqcup_{j \in J} F(R^j)$ to:
\[   \begin{pmatrix}
p' (z)& b & 0 \\
1 &  0   & 0 \\
    q' (z) & 0 &  \lambda
\end{pmatrix}  \qand \begin{pmatrix}
0 & 1 & 0 \\
1/b &  - p' (w) / b  & 0 \\
 0 &  - q' (w) & 1 / \lambda
\end{pmatrix}
 \]
where $p' $  is close to $1/\eta=10^4$ and both $q' $ and $b$ are small.  The result follows. 
\end{proof}

In the following result, we introduce the two hyperbolic basic sets whose continuations will be later involved in the formation of heterodimensional cycles.

\begin{proposition} \label{polynomeq} 
There exists a neighborhood $\mathcal{U}_1$ of $F_1$ in $\mathrm{Aut}(\mathbb{C}^3)$ such that every  $F \in   \mathcal{U}_1$  displays a    hyperbolic basic set
 \[ K = \bigcap_{n \in \mathbb{Z}} F^n ( \bigsqcup_{j \in J} R^j ) \] 
 of index $2$ included in $\mathbb{D}(0,1/2)^2 \times \mathbb{D}$  and a saddle fixed point $S$ of index $1$ close to $(3,3,3)$.  
\end{proposition}

\begin{proof}
{\bf We first show that $K_1 = K(F_1)$ is a hyperbolic basic set of index $2$.}
The set $K_1$ is clearly compact, $F_1$-invariant  and  locally maximal. By \cref{coneproperties} and by the cone field criterion, the set $K_1$ is also hyperbolic with unstable (resp. stable) dimension equal to  $2$ (resp. $1$). 
Finally, let us show that $K_1$ is  transitive. 

\begin{fact}
For every $(z,w) \in K_o$, the set $K_1$ intersects the line $\{ (z,w,t) \,  | \, t \in \mathbb{C} \}$ at exactly one point $(z,w,\xi(z,w))$, and the map $(z,w) \in K_o \mapsto \xi(z,w) \in \mathbb{C}$ is continuous.
\end{fact}

\begin{proof} 
We fix $(z,w) \in K_o$. Let us denote for every $n \in \mathbb{Z}$ by $I_n$ the third coordinate projection of $F_1^n ( \{H^{-n}(z,w)\} \times \mathbb{D})$. We notice that $I_n$ is a disk of diameter $\lambda^n$. Also $F_1 ( \{H^{-n-1}(z,w)\} \times \mathbb{D})$ is close to  $ \{H^{-n}(z,w)\} \times  ( \lambda \cdot \mathbb{D} +  \frac{1}{10}  e^{i \frac{j \pi}{2}} )$ for some $j \in J$, with $( \lambda \cdot \mathbb{D} +  \frac{1}{10}  e^{i \frac{j \pi}{2}} ) \Supset \mathbb{D}$. Therefore we have $I_n \Subset I_{n+1}$ for every $n \in \mathbb{Z}$. Thus $\bigcap_{n \in \mathbb{Z}} I_n$ is reduced to a point $\xi(z,w) \in \mathbb{D}$. But $\{ (z,w) \} \times \bigcap_{n \in \mathbb{Z}} I_n$ is precisely equal to the intersection of $K_1$ with $\{ (z,w,t) \,  | \, t \in \mathbb{C} \}$.  

Then the continuity of $\xi$ follows from the continuity of $F_1$.    
\end{proof}

We recall that $K_o$ is transitive for the H\'enon map $H$. 
We can therefore pick $(z,w) \in K_o$ such that $\{ H^n(z,w) \, | \, n \ge 0 \}$ is dense in $K_o$. Then the sequence $\{  F^n_1(z,w,\xi(z,w)) \, | \, n \ge 0 \}$ is dense in $K_1$. Indeed, for every $(\alpha,\beta,\xi(\alpha,\beta)) \in K_1$, we can find $H^n(z,w)$ arbitrarily close to $(\alpha,\beta)$. Since $\xi$ is continuous, then $\xi(H^n(z,w))$ is close to $\xi(\alpha,\beta)$. Thus $F_1^n(z,w,\xi(z,w)) = (H^n(z,w), \xi(H^n(z,w)) $ is close to $(\alpha,\beta,\xi(\alpha,\beta))$. This shows that $K_1$ is transitive for $F_1$. 

 This shows that  $K_1$ is a  hyperbolic basic set of index $2$ for $F_1$. By robustness of hyperbolic basic sets, the set $K$ is still a hyperbolic basic set of index $2$ for $F$ for every $F$ close to $F_1$. \\

{\bf Now we show that $F$ displays a saddle  fixed point $S$ of index $1$.} 
Recall that $H$ displays an attracting fixed point $S_o = (S_o^1 , S_o^2)$  close to $(3,3)$ by \cref{henon}. Since $q$ is close to $-1/3$ on $\mathbb{D}_4$, the number $t_1 := (1-\lambda)^{-1} \cdot q(S_o^1)$ is close to $-9 \cdot (-1/3) = 3$   and $S_1 = (S_o , t_1)$ is a  fixed point for $F_1$ close to $(3,3,3)$. Moreover the differential of $F_1$ at $S_1$ is close to: 
\[\begin{pmatrix}
p' (3)& b & 0 \\
1 & 0  & 0 \\
    q' (3) & 0 &  \lambda
\end{pmatrix}
\]
where $p' (3)$  is close to $\eta = 10^{-4}$ and both $q' (3)$ and $b$ are small. Thus $D_{S_1} F_1 $ has two eigenvalues 
smaller than 1 in modulus and one larger than 1 in modulus. Thus $F_1$ displays a  saddle fixed point $S_1$ of index $1$ close to $(3,3,3)$. The result follows by robustness of hyperbolic basic sets. 
\end{proof}

\begin{definition}
For every $k \ge 1$ and  $\sw =( j_0  \cdots    j_k) \in J^k$, we define the following set:
\[ R^\sw := \bigcap_{0 \le n \le k} F^{-n}(R^{j_{n}}) \, .  \]
\end{definition}

We also denote:
\[ \partial_z R^\sw := R^\sw \cap  F^{-k-1}( \partial \mathbb{D} \times \mathbb{D}^2) \, ,  \quad  \partial_w R^\sw := R^\sw \cap  ( \mathbb{D} \times \partial \mathbb{D} \times \mathbb{D})  \qand \partial_t R^\sw := R^\sw \cap F^{-k-1}( \mathbb{D}^2 \times \partial \mathbb{D}) \, .  \]

\begin{proposition} \label{locmfd}
For every $\ss = (j_k)_{k \ge 0} \in J^\N$, the set $W^\ss := \bigcap_{ n \ge  0}  R^{ (j_0, \cdots , j_n)}$ is a local stable manifold of $K$ which is a holomorphic graph over $w \in  \mathrm{int}  \,  \mathbb{D}$ with tangent spaces  in $C^{s}$.

For every $\su  = ( j_{k})_{k<0}  \in J^{\Z_-^*}$, the set $W^\su := \bigcap_{ n>0} F^n( R^{ ( j_{-n},  \cdots , j_{-1})  })$ is a local unstable manifold of $K$ which is a holomorphic  graph over $(z,t) \in (  \mathrm{int}  \,  \mathbb{D})^2$ with tangent spaces  in $C^{u}$. 
\end{proposition}

\begin{proof}
By \cref{coneproperties}, for every $\ss \in J^\N $, the set $W^\ss$ intersects each $\{ w = w_o \}$, with $w_o \in \mathbb{D}$, at exactly one point depending continuously on $w_o$. Thus $W^{\ss}$ is a continuous graph over $w \in \mathbb{D}$.

Let $M \in W^\ss$. For every $n \ge 0$, notice that $F^n(M)$ belongs to the set $F^n(R^{ (\ss_1, \cdots , \ss_n)}) \cap R^{ (\ss_{n+1} , \cdots , \ss_{2n})} $. The latter is included in $\bigcap_{-n \le k \le n } F^k ( \bigsqcup_{j \in J} R^j ) $, which is a small neighborhood of $K$. Thus the distance of $F^n(M)$ to $K$ is small when $n$ is large and so $W^\ss$ is a local stable manifold of $K$. In particular it is a holomorphic graph. Still by \cref{coneproperties}, its tangent spaces are in $C^s$. 

The same proof is working for unstable manifolds. 
\end{proof}

We also define:
\[ W^s_{loc} (K) = \bigcup_{ \ss \in J^\N   } W^\ss \qand   W^u_{loc} (K) = \bigcup_{ \su \in   J^{\Z_-^*}  }  W^\su  \, . \] 

The saddle point $S$ and its stable/unstable manifolds will be of particular interest to us:

\begin{definition} 
The local stable (resp. unstable) manifold $W^s_{loc}(S)$ (resp. $W^u_{loc}(S)$) of $S$ is the connected component of $W^s(S) \cap \mathbb{D}(3,1)^3$ (resp. $W^u(S) \cap  \mathbb{D}(3,1)^3$) containing $S$.  
\end{definition}

Up to restricting the neighborhood $\cU_1$ of $F_1$ defined by \cref{polynomeq}, we obtain:

\begin{fact}  \label{locsmfdS}
The local stable manifold $W^s_{loc}(S)$ of the saddle point $S$  is  $C^1$-close to:
\[\{ (z, w , 3 ) \, | \, z \, , w \in  \mathbb{D}(3 ,1)     \} \, .  \] 
\end{fact} 

\begin{proof}  
This is an immediate consequence of the hyperbolic continuation of local stable manifolds (indeed, for $F=F_1$ with $q=-1/3$ on  $\mathbb{D}_4$, $W^s_{loc}(S_1)$ is simply equal to $\{ (z , w , 3) \, | \, \, z \, , w \in  \mathbb{D}(3,1)   \}$ since the basin of the sink $S_o$ of the H\'enon map $H$ contains $\mathbb{D}(3,1)^2$ by \cref{henon}). 
\end{proof}

\begin{fact} \label{factunsts}
The local unstable manifold $W^u_{loc}(S)$ of the saddle point $S$  is  $C^1$-close to:
\[\{ (3 , 3    ,  3 +  t ) \, | \, t \in   \mathbb{D}    \}  \, . \] 
\end{fact} 

\begin{proof}
This is an immediate consequence of the hyperbolic continuation of local unstable manifolds (indeed, for $F=F_1$ with $q=-1/3$ on  $\mathbb{D}_4$, $W^u_{loc}(S_1)$ is simply equal to $\{ (3 , 3 , 3 +  t ) \, | \, \,t \in  \mathbb{D}  \}$). 
\end{proof}

\begin{remark}  \label{remU1}
The set  $\mathcal{U}_1$ can be taken as the $\epsilon_1$-$C^0$-neighborhood of $F_1$ on some small neighborhood of $W^s_{loc} (K_1) \cup  F_1^{-1} (W^u_{loc} (K_1)) \cup W^s_{loc} (S_1) \cup F_1^{-1} (W^u_{loc} (S_1))$, for some small $\epsilon_1 > 0$. 
\end{remark}

\section{Blender property}

In this section, we show that the compact hyperbolic set $K$ for the automorphism $F  \in   \mathcal{U}_1$ displays the  blender property.
We first define the following constant cone field:
\[ C^{uu} = \{v = (v_{1},v_{2},v_{3}) \in \mathbb{C}^{3} : \| (v_{2},v_{3}) \|_2  \le     10^{-3} \cdot   |v_{1}|\} \, .     \]

\begin{lemma} \label{coneproperties2}
For every  $F  \in   \mathcal{U}_1$,   any  non zero vector in   $C^{uu}$  is sent into $\mathrm{int} \, C^{uu}$ and expanded by a factor  larger than $10^3$ by $D_{ (z,w,t) } F$ for any  
 $(z,w,t) \in \bigsqcup_{j \in J} R^j$.
\end{lemma} 

\begin{proof}
The proof is exactly the same  as in the proof of \cref{coneproperties}. 
\end{proof}

\begin{definition}
A $uu$-curve  is a continuous graph $\cC = \{ (z , c_2 (z), c_3 (z) ) \, | \,     z \in \mathbb{D} \} \subset  \mathbb{D} \times (\mathrm{int}  \, \mathbb{D} )^2 $  such that  $z \mapsto (c_2(z), c_3(z))$ is holomorphic on $\mathrm{int} \, \mathbb{D}$ and $ (1, c'_2 (z), c'_3 (z) ) \in C^{uu}$ for every $z \in \mathrm{int}  \,  \mathbb{D}$.
\end{definition} 

Here is the main result of this section. It is the aforementioned blender property.

\begin{proposition} \label{robustmain}  
Every $uu$-curve $\cC \subset \mathbb{D}^2 \times \mathbb{D}(0,1/2)$ intersects $W^s_{loc}(K)$ for every   $F  \in   \mathcal{U}_1$.   
\end{proposition}

The key to prove \cref{robustmain} is the following one-dimensional open covering property stated in  \cref{sublemme2}. Let us first  define the four following affine maps:  
 \[ L_0 (t) := \frac{9}{10} t - \frac{9}{100} \, ,     \hspace{0.5cm}  L_1 (t) := \frac{9}{10} t - \frac{9i}{100}   \, ,    \hspace{0.5cm} 
 L_2 (t) := \frac{9}{10} t + \frac{9}{100}    \qand  L_3 (t) := \frac{9}{10} t + \frac{9i}{100}   \,  .  \]  
 We remark that the third coordinate of $F^{-1}$ is close to $L_j$ on $F(R^j)$.

\begin{sublemma}  \label{sublemme2}
For any $t \in \mathbb{D}(0,1/2)$, there exists $\tau \in \mathbb{D}(0, 0.497)$ and $j \in J$ such that $t = L_j(\tau)$. 
\end{sublemma}

\begin{proof} 
By symmetry, it suffices to show that $\{ r e^{i \theta} \, | \,  0 \le r \le  0.5 \, , \,  -\pi/4 \le \theta \le \pi/4 \}$ is included in the image of  $\mathbb{D}(0, 0.497)$ by $L_2$. 
For every $0 \le r \le  0.5$ and $-\pi/4 \le \theta \le \pi/4$, we have $ r e^{i \theta}   \in L_2(\mathbb{D}(0, 0.497))$ if and only if $| r e^{i \theta}  - 0.09|  \le 0.9 \cdot 0.497$, that is if and only if:
\[ \sqrt{ r^2 - 0.18 \cdot r \cos(\theta)  + 0.09^2   }  \le  0.9  \cdot 0.497 \,   .  \]
Rapid computations show this is true if the inequality is satisfied for $r = 0.5$ and $\theta = \pi/4$, which is the case since  $0.5^2 - 0.18 \cdot  0.5 \cdot \sqrt{2}/2 + 0.09^2 < 0.2 <  0.9^2  \cdot 0.497^2$.
\end{proof}

We are now ready to prove \cref{robustmain}. 

\begin{proof}[Proof  of \cref{robustmain}]
It is enough to show that for every  automorphism  $F  \in   \mathcal{U}_1$, there is $\ss  = (j_0,  j_1, \cdots ) \in J^\N$ such that for every $n \ge 0$, the curve $\cC$ intersects $R^{(j_{1}, \cdots  , j_{n})}$. Then $\cC$  has a nonempty intersection with the local stable manifold $W^\ss =  \bigcap_{ n \ge 0} R^{(j_{0},  \cdots , j_{n})}$. The result then follows immediately by induction from the following \cref{robust}.
\end{proof}

\begin{lemma} \label{robust} 
Let $\cC \subset \mathbb{D}^2 \times \mathbb{D}(0,1/2)$ be a $uu$-curve. Then  for every  $F  \in   \mathcal{U}_1$,  there is $j \in J$ such that $\cC$ intersects $R^j$ and $F(\cC  \cap R^j)$ is a $uu$-curve included in $\mathbb{D}^2  \times \mathbb{D}(0,1/2)$. 
\end{lemma}

\begin{proof}[Proof of \cref{robust} ] We start with the following intermediate result:

\begin{sublemma} \label{sublemme1}
Let $\cC$ be a $uu$-curve included in $\mathbb{D}^2 \times \mathbb{D}(0,1/2)$. Then  for every $F  \in   \mathcal{U}_1$, for any $j \in J$, the curve $\cC$ intersects the interior of $R^j$ but does not intersect $\partial_w R^j \cup \partial_t R^j$.
\end{sublemma}

\begin{proof}
 Let us denote  $z_j =1/4$ if $j=0$, $=i/4$ if $j = 1$, $=-1/4$ if $j =2$ and $=-i/4$ if $j = 3$. We remark that $R^j$ contains the point $(z_j,c_2 (z_j),c_3 (z_j) ) \in \cC$ in its interior since $c_2(z_j) \in \mathrm{int}  \,  \mathbb{D}$ and  $c_3 (z_j) \in \mathbb{D}(0,1/2)$ while $\{ z_j \} \times \mathrm{int}  \,  \mathbb{D}  \times  \mathbb{D}(0,1/2) \subset  \mathrm{int}  \, R^j$ (indeed $R^j$ is close to $\mathbb{D}_j \times \mathbb{D} \times  L_j  ( \mathbb{D})$ with $L_j  ( \mathbb{D}) \Supset \mathbb{D}(0,1/2)$). Therefore $\cC$ intersects $ \mathrm{int}  \, R^j$.  Also since $\cC \subset  \mathbb{D} \times \mathrm{int}  \, \mathbb{D}   \times \mathbb{D}(0,1/2)$ while $\partial_w R^j \subset  \mathbb{D} \times \partial  \mathbb{D} \times  \mathbb{D}$ and $ \partial_t R^j \cap (\mathbb{D}^2   \times \mathbb{D}(0,1/2))  = \emptyset $, the result follows.
\end{proof}

\begin{sublemma}  \label{sublemme3}
Let $\cC$ be a $uu$-curve intersecting the interior of $R^j$, for  $j \in J$, but not $\partial_w R^j \cup \partial_t R^j$. Then $F(\cC  \cap R^j)$ is a $uu$-curve. 
\end{sublemma}

\begin{proof}
It is immediate from \cref{coneproperties2} that $F(\cC \cap R^j)$ has its tangent spaces in $C^{uu}$. Moreover, if we denote by $\pi_z$ the first coordinate projection, the map $\psi : M \in  \cC \cap R^j  \mapsto \pi_z ( F (M) ) \in \mathbb{C}$ is open (since holomorphic and non constant) on the interior of $ \cC \cap R^j $. Thus $\partial  \, \mathrm{Im} \psi$ is included in $\psi ( \partial (   \cC  \cap  R^j )) = \psi (   \cC  \cap  \partial_z  R^j ) \subset \partial \mathbb{D}$.   Since $ \mathrm{Im} \psi \subset \mathbb{D}$, we conclude that $\partial  \, \mathrm{Im} \psi = \partial \mathbb{D}$. Thus $F(\cC  \cap R^j)$ is a graph over $z \in \mathbb{D}$ and so a $uu$-curve.  
\end{proof}

Let us now conclude the proof of \cref{robust}.
By \cref{sublemme2},  there exist $\tau \in \mathbb{D}(0, 0.497)$ and $j \in J$ such that $ c_3 (0)= L_j(\tau)$. Also, by \cref{sublemme1}, the curve    $\cC$  intersects  $R^j$ but  not $\partial_w R^j \cup \partial_t R^j$.  Then, by \cref{sublemme3}, we have that  $F(\cC  \cap R^j)$ is a $uu$-curve. Since any point  $ (z,c_2(z),c_3(z)) \in  \cC  \cap  R^j$ has its third coordinate $c_3(z)$ which is $10^{-3}$-close to $c_3 (0)$, the preimage of $c_3(z)$ by $L_j$ is in $\mathbb{D}(0, 0.497 + \lambda \cdot 10^{-3}) \subset \mathbb{D}(0, 0.499)$. Since the third coordinate of $F^{-1}$ is close to $L_j(t)$ on $F(R^j)$, it follows that $F(\cC  \cap R^j)$ is a   $uu$-curve included in $\mathbb{D}^2  \times \mathbb{D}(0,1/2)$.  
 \end{proof}

\section{Creating robust cycles}

 We now create robust heterodimensional cycles and conclude the proof of the main Theorem. We will take the composition of $F_1$ with finitely many polynomial automorphisms of the form:
\[(z,w,t) \mapsto (z,w,t + P(z,w)) \text{ or } (z,w,t) \mapsto (z,w + P(z,t),t) \text{ or } (z,w,t) \mapsto (z+P(w,t),w,t) \]
where $P :  \mathbb{C} \rightarrow \mathbb{C}$ is a polynomial.  
So the final map will also be a polynomial automorphism.

\subsection{Robust intersections between $W^s(S)$ and $W^u(K)$}

We first create a robust intersection between the stable manifold $W^s(S)$ of the continuation $S$ of the saddle point $S_1$ and the unstable set $W^u(K)$ of the continuation $K$ of the basic set $K_1$, by perturbating $F_1$. This is the easiest part since the intersection will be transverse and hence robust. 

\begin{proposition} \label{prop1}
There exists $F_2 \in \mathcal{U}_1$ and a neighborhood $\cU_2 \subset \cU_1$ of $F_2$ such that every $F \in   \mathcal{U}_2$  
displays a nonempty intersection between $W^s(S)$ and $W^u(K)$. 
\end{proposition}

\begin{proof}[Proof of \cref{prop1}]
Let $A_1 \in K_1$ be a fixed point of $F_1$. 
Let us take $\omega>0$ small.
We set $\gamma_1:=1 + 2 \omega$ and $\alpha'_1:=3.9$. 
The following is an immediate consequence of \cref{locmfd} and \cref{locsmfdS} applied to $F=F_1$:

\begin{fact}
There exist $M_1  =  (\gamma_1,\beta_1, \gamma_1)    \in  F_1(W^u_{loc}(A_1))$ and $N_1  = (\alpha'_1 ,\alpha'_1,\gamma'_1) \in W_{loc}^s(S_1)$. 
\end{fact} 
 
Let $t \in \mathbb{C} \mapsto P_1(t)$ be a polynomial close to 1 on $\mathbb{D}(\gamma_1,\omega)$ and  to 0 on $\mathbb{D} \sqcup \mathbb{D}(3,1)$. 
Let $z \in \mathbb{C} \mapsto P_2(z)$ be a polynomial close to 1 on $\mathbb{D}(\alpha'_1,\omega)$ and close to 0 on $\mathbb{D} \sqcup \mathbb{D}(3,1/2)$. Such polynomials exist by the Runge theorem. We set:
\[g_1(z,w,t) := (z + P_1(t) \cdot (\alpha'_1 -  \gamma_1 ),w,t)  \, , \]
\[g_2(z,w,t):=   (z,w+P_2(z) \cdot  (\alpha'_1-\beta_1),t)  \, ,   \]
\[ \qand g_3(z,w,t):= (z,w,t+P_2(z) \cdot (\gamma'_1- \gamma_1 )) \, .  \]
Notice that $g_3 \circ g_2 \circ g_1$ sends $M_1$ close to $N_1$. Now we set:
\[ F_2 :=   g_3 \circ g_2 \circ g_1  \circ  F_1 \, .  \]

\begin{fact}  \label{factF2U1}
The map $F_2$ belongs to the neighborhood $\cU_1$ defined in \cref{polynomeq}  and \cref{remU1}.
\end{fact}

\begin{proof}
The first and third coordinate projections of $F_1(W^s_{loc} (K_1)) \cup  W^u_{loc} (K_1) \cup F_1(W^s_{loc} (S_1)) \cup W^u_{loc} (S_1)$ are respectively included in $\mathbb{D} \sqcup \mathbb{D}(3,1/2)$ and $\mathbb{D} \sqcup \mathbb{D}(3,1)$. Therefore the maps $g_1$, $g_2$ and $g_3$ are all close to the identity on a neighborhood of the latter  set and hence $F_2 :=   g_3 \circ g_2 \circ g_1  \circ  F_1 $ is close to $F_1$ on   $W^s_{loc} (K_1) \cup  F_1^{-1} (W^u_{loc} (K_1)) \cup W^s_{loc} (S_1) \cup F_1^{-1} (W^u_{loc} (S_1))$.
\end{proof}

In particular $F_2$ displays the hyperbolic continuations $K_2$ of $K_1$ and $S_2$ of $S_1$.  Also by \cref{locsmfdS}, there is a bidisk $\Delta^s$ embedded in $W_{loc}^s(S_2)$ which is $C^1$-close to  
 \[ N_1 + \{ (z , w  , 0 ) \, | \, z  ,  \, w \in \mathbb{D}(0,\omega) \}  \, . \] 
 Let $A_2$ be the continuation of $A_1$ for $F_2$. 
 By \cref{locmfd} and since $\omega$ is small, there is a bidisk $\Delta^u$ embedded in $F_1(W^u_{loc}(A_2))$ which is a graph of the form:
  \[ \Delta^u =    M_1 +    \{ ( z , w(z ,t)  , t ) \, | \, z    , \, t \in \mathbb{D}(0,\omega)   \} \, ,  \]
  where $w(0,0)$ is small and  the tangent spaces of $\Delta^u$ are included in a cone slightly larger than $C^u$ (in particular $w(z,t) \in \mathbb{D}(0,\omega)$ for $(z,t) \in \mathbb{D}(0,\omega)^2$). 
    Since $M_1$ is sent close to $N_1$ by $g_3 \circ g_2 \circ g_1$ and  the restriction of $g_3 \circ g_2 \circ g_1$ to $M_1+\mathbb{D}(0,\omega)^3$ is close to a translation, the set $(g_3 \circ g_2 \circ g_1)(\Delta^u)$ (which is included in  $F_2(W^u_{loc}(A_2)) \subset W^u(A_2)$) is $C^1$-close to: 
 \[   N_1 +    \{ (z , w(z , t)  , t ) \, | \, z , \, t \in  \mathbb{D}(0,\omega)   \} \, . \]
 The directions of the tangent spaces  of $\Delta^s$ are close to $\mathbb{C} \cdot (1,0) + \mathbb{C} \cdot (0,1)$ while those of $(g_3 \circ g_2 \circ g_1)(\Delta^u)$  are included in a cone slightly larger than $C^u$. 
Therefore $\Delta^s \subset W_{loc}^s(S_2)$ and $(g_3 \circ g_2 \circ g_1)(\Delta^u) \subset W^u (A_2)$ have a  transverse intersection close to $N_1 + (0,w(0,0),0) \approx N_1$. In particular, this intersection is robust, which ends the proof of \cref{prop1}. 
\end{proof}

\begin{remark}  \label{remU2}
The set  $\cU_2$ can be taken as the subset of $\cU_1$ formed by maps which are $\epsilon_2$-$C^0$-close to $F_2$ on the $\omega$-neighborhood of  $F_1^{-1}(M_1)$, 
 for some small $\epsilon_2 > 0$. 
\end{remark}

\subsection{Robust intersections between $W^u(S)$ and $W^s(K)$}

We now create a robust intersection between the unstable manifold $W^u(S)$ of the continuation $S$ of the saddle point $S_2$ and the stable set $W^s(K)$ of the continuation $K$ of the basic set $K_2$. This is the difficult part since both  the unstable manifold  of  $S$ and the stable manifolds of points of $K$ are one-dimensional so we can not apply a transversality argument here. To overcome this difficulty, we will use here  the blender property to obtain the robustness.

\begin{proposition} \label{prop2}
There exist $F_3 \in \mathcal{U}_2$ and a neighborhood $\mathcal{U} \subset \mathcal{U}_2$ of $F_3$ such that every $F \in   \mathcal{U}$ displays a nonempty intersection between $W^u(S)$ and $W^s(K)$. 
\end{proposition}

\subsubsection{Obtaining an initial intersection between $W^u(S)$ and $W^s(K)$. }

We now perturb the automorphism $F_2$  into $F_3$ in order to get  an  intersection point between $W^u(S_3)$ and $W^s(K_3)$.  We start by selecting a particular periodic point in $K_2$:

\begin{fact}  \label{factA3}
There is a periodic point $B_2 \in K_2$ of $F_2$ of third coordinate in $\mathbb{D}(0,1/10)$. 
\end{fact} 

\begin{proof} 
By \cref{robustmain}, there is a point in $W^s_{loc}(K_2)$ of third coordinate 0. Since by \cref{locmfd} any local stable manifold has its tangent spaces  in $C^{s}$, there is a point of $K_2$ of third coordinate close to 0. Then, by density of periodic points in $K_2$,  there is a periodic point $B_2 \in  K_2$ of third coordinate in  $\mathbb{D}(0,1/10)$.
\end{proof}

 Let us take again $\omega>0$ small.  We set $\gamma_2 := 2 - 2 \omega$ and $\beta'_2 :=  - 0.9 $. 
The following is an immediate consequence of \cref{locmfd} and \cref{factunsts} together with \cref{factF2U1}:

\begin{fact}
There exist $M_2  = (\alpha_2,\beta_2, \gamma_2) \in F_2(W^u_{loc}(S_2))$ and $N_2  = (\alpha'_2, \beta'_2,\gamma'_2) \in W^s_{loc}(B_2)$. 
\end{fact}

Let $t \in \mathbb{C} \mapsto Q_1(t)$ be a polynomial close to 1 on $\mathbb{D}(\gamma_2,\omega)$ and close to 0 on $\mathbb{D} \sqcup \mathbb{D}(3,1)$. 
 Similarly let $w \in \mathbb{C} \mapsto Q_2(w)$ be a polynomial close to 1 on $\mathbb{D}(\beta'_2,\omega)$ and close to 0 on $\mathbb{D}(0,1/2)  \sqcup \mathbb{D}(3,1) $. 
 Let us also take $\mu \in \mathbb{D}(0,1/10)$.
 Then we define: 
\[h_1(z,w,t) := (z  ,w + Q_1(t) \cdot (\beta'_2-\beta_2)    ,t) \, ,  \]
\[h_2(z,w,t):=(z,w,t+Q_2(w) \cdot (\gamma'_2-\gamma_2))    \,    ,   \]
\[  \qand  h_3(z,w,t):=(z+ Q_2(w) \cdot   (\alpha'_2 - \alpha_2+\mu + (t-\gamma'_2))   , w    ,t )    \,  .    \]
  Notice that $h_3 \circ h_2 \circ h_1$ sends $M_2$ close to $N_2$. Now we set: 
\[ F_3 :=  h_3 \circ h_2 \circ h_1  \circ    F_2 \, .  \]

\begin{fact} 
The map $F_3$ belongs to the neighborhood $\cU_2$ defined in \cref{prop1}  and \cref{remU2}.
\end{fact}

\begin{proof} 
The same proof as for \cref{factF2U1} shows that $F_3 \in \cU_1$. Also recall that $F^{-1}_1(M_1)$ is sent close to $N_1$ by $F_2$. 
Since $N_1 = (3.9,3.9, \gamma'_1)$ with $\gamma'_1 \in  \mathbb{D}(0,1/2)$, the maps $h_1$, $h_2$ and $h_3$ are all close to the identity on a ball around  $N_1$ of radius bounded from below independently of $\omega$. Since $\omega$ is small, this implies that $F_3$ is close to  $F_2$ on the $\omega$-neighborhood of  $F_1^{-1}(M_1)$ and so $F_3 \in \cU_2$. 
\end{proof}

By \cref{locmfd}, there is a disk $U_\mu$ embedded in the local stable manifold of the continuation $B_3 \in K_3$ of $B_2 \in K_2$ for $F_3$ which is close to: 
\begin{equation} \label{angle} 
 \{   ( \alpha'_2 ,\beta'_2   + w  ,\gamma'_2) \, | \,  w \in  \mathbb{D}(0,\omega)  \} \, . 
 \end{equation}
  Let $S_3$ be the continuation of $S_2$ for $F_3$. 
By \cref{factunsts}  and since $\omega$ is small, there is a disk $\tilde{V}_\mu$ embedded in $F_2(W^u_{loc}(S_3))$ which is close to $\{   (\alpha_2  , \beta_2    ,   \gamma_2  + t ) \, | \, t \in  \mathbb{D}(0,\omega)  \}$. Then its image  is  $V_\mu := ( h_3 \circ h_2 \circ h_1)(\tilde{V}_\mu) \subset F_3(W^u_{loc}(S_3)) \subset W^u (S_3)$ and  is close to: 
\begin{equation} \label{angle2} 
 \{   ( \alpha'_2      + \mu + t   , \beta'_2 ,\gamma'_2 + t ) \, | \, t \in  \mathbb{D}(0,\omega)  \} \, . 
 \end{equation}
 
We will use the following independent technical result:

\begin{lemma} \label{selectionparameter}
Let $(U_\mu)_{\mu \in \mathbb{D}}$ be a holomorphic family of   
graphs $w \in \mathrm{int} \,  \mathbb{D} \mapsto (U^1_\mu (w) , w , U^3_\mu (w) ) \in \mathbb{C}^3$, where $U^1_\mu$ and $U^3_\mu$ are small.  
Let $(V_\mu)_{\mu \in \mathbb{D}}$ be a holomorphic family of   graphs $t \in  \mathrm{int} \, \mathbb{D} \mapsto  (V^1_\mu (t) ,  V^2_\mu (t)  ,    t  ) \in \mathbb{C}^3$. 
Suppose that $\mu   \mapsto V^1_\mu(0)$ is close to $\mu$ and that $V^2_\mu (t) \in \mathrm{int} \,  \mathbb{D}$.
Then there exists $\mu_o \in \mathbb{D}$ such that $U_{\mu_o}$ intersects $V_{\mu_o}$. 
\end{lemma}

\begin{proof} 
Up to the change of coordinates $(z,w,t) \mapsto (z - U^1_\mu (w), w , t - U^3_\mu (w))$ which is  close to the identity, we can suppose that $U_\mu = \{ (0, w ,0) \, | \,  w  \in \mathbb{D} \}$. Notice that $V_\mu$ intersects the plane $\{ t = 0\}$ at a unique point equal to $ (V^1_\mu (0) ,  V^2_\mu (0)  , 0 )$. Also the map $\mu \in  \mathrm{int} \,  \mathbb{D} \mapsto V^1_\mu (0)$ is holomorphic and close to $\lambda \mapsto  \lambda  $.  Then by the Rouch\'e theorem there is $\mu_o \in \mathbb{D}(0,1/2)$ such that $V^1_{\mu_o} (0) = 0$. Then the point $(0, V^2_{\mu_o}  (0)  , 0) $ belongs to $U_{\mu_o} \cap  V_{\mu_o}$. 
\end{proof}

As a consequence of  \cref{selectionparameter} together with  \cref{angle} and \cref{angle2}, we obtain:

\begin{proposition} \label{constantes}
 The unstable manifold $W^u( S_3 )$ intersects the local stable manifold  $W^s_{loc}(B_3)$ at a point $N$  such that $T_N W^u( S_3 )$ is close to $\mathbb{C} \cdot ( 1  , 0 , 1)$. 
\end{proposition}

\subsubsection{Obtaining robust intersections between $W^u(S_3)$ and $W^s(K_3)$ using the blender property. }

We now conclude the proofs of \cref{prop2} and of  the main Theorem. We use the blender property to make robust the initial intersection between $W^u(S_3)$ and $W^s(K_3)$ we obtained above.

\begin{proposition} \label{iteree}
Let $\Delta_o$ be an embedded  open disk intersecting $W^s_{loc}(B_3)$ at a point $N$ such that $T_N \Delta_o$ is close to $\mathbb{C} \cdot (1,0,1)$. Then there exists $n \ge 1$ such that for every  map $F$ close to $F_3$ and every  embedded disk $\Delta$ close to $\Delta_o$, the intersection  $F^n(\Delta) \cap (\mathbb{D}^2 \times       \mathbb{D}(0,1/2))$ contains a $uu$-curve. 
\end{proposition}

\begin{proof}
Let us consider the following matrix:
\[\mathcal{M}_o := \begin{pmatrix}
10^4 & 0 & 0 \\
1 & 0  & 0 \\
    0 & 0 &  \lambda
\end{pmatrix} 
\]
For every matrix $\mathcal{M}$ close to $\mathcal{M}_o$ and every vector $v$ close to $(1,0,1)$, we have $ \mathcal{M} \cdot v \in \mathrm{int} \,  C^{uu}$. Observe that the differential of $F_3$ at any point in $W^s_{loc}(B_3)$ is  close to $\mathcal{M}_o$. 
 Thus, if $\Delta_o$ intersects $W^s_{loc}(B_3)$ at $N$ with a  direction close to $\mathbb{C} \cdot (1,0,1)$, the direction of $F_3(\Delta_o)$ at $F_3(N) \in W^s_{loc}(B_3)$ is in $\mathrm{int} \,  C^{uu}$. Then, by invariance of the cone $C^{uu}$ (see \cref{coneproperties2}) and by  expansion in the strong unstable direction, for $n$ large enough,
  $F_3^{n}(\Delta_o)$ contains a holomorphic graph $\Delta_n$ over $z$ varying in a neighborhood of $\mathbb{D}$ with tangent spaces in  $\mathrm{int} \, C^{uu}$ (in particular it contains a $uu$-curve). 
  If moreover $n$ is a multiple of the period of $B_3$, then $\Delta_n$ intersects $W^s_{loc}(B_3)$ at a point $F_3^{n}(N)$ close to $B_3$. Since the third coordinate of $B_3$ is in $\mathbb{D}(0,1/10)$ by \cref{factA3}, the curve $\Delta_n$ is included in $\mathbb{C}^2 \times  \mathbb{D}(0,1/3)$.   Then, by continuity,  for every  map $F$ close to $F_3$ and every $\Delta$ close to $\Delta_o$, we have that $F^n(\Delta) \cap (\mathbb{D}^2 \times       \mathbb{D}(0,1/2))$ contains a $uu$-curve. 
\end{proof}

\begin{proof}[Proof of \cref{prop2}]
By \cref{constantes} and then \cref{iteree},  for every $F$ in a neighborhood $\cU \subset \cU_2$ of  $F_3$,  the unstable manifold  $W^u( S )$ contains a $uu$-curve $\cC \subset \mathbb{D}^2 \times  \mathbb{D}(0,1/2)$. Therefore by \cref{robustmain}, there is a nonempty intersection between  $\cC \subset W^u(S)$  and $W_{loc}^s(K)$.
This ends the proof of \cref{prop2}. 
\end{proof}

\begin{proof}[Proof of the Main Theorem]
It follows immediately from \cref{prop1} and \cref{prop2}.
\end{proof}

\begin{proof}[Proof of the main Corollary]
We just take $\cU_d := \cU \cap \mathrm{Aut}_d(\mathbb{C}^3)$ with $d:=\mathrm{deg} \, F_3$ to conclude. 
\end{proof}

\small \noindent \texttt{sebastien.biebler@imj-prg.fr} \\
\small Sorbonne Universit\'e, Universit\'e Paris-Cit\'e, CNRS  \\
Institut de Math\'ematiques de Jussieu-Paris Rive Gauche \\
 75005 Paris, France \\

\end{document}